\documentclass[a4paper,11pt]{amsart}

\usepackage{amsmath,amssymb,amsthm,bm,mathtools,xspace,booktabs,url}
\usepackage{color,tikz}
\RequirePackage[colorlinks,citecolor=blue,urlcolor=blue]{hyperref}

\title[A Factorization of a L\'evy Process over a Phase-Type  Horizon] {A Factorization of a L\'evy Process \\
over a Phase-Type  Horizon}

\author[S.\ Asmussen \and J.\ Ivanovs]{S{\o}ren Asmussen \and Jevgenijs Ivanovs}

 \newcommand{\CoxFig}{
\newcommand{\ath}{thick}  \newcommand{\bth}{[thick]}
\newcommand{\drct}{\filldraw[fill=blue!20!white,thick]}
  \begin{figure}[htbp]
   \centering 
\begin{tikzpicture}[scale=0.61] 
\footnotesize
\drct (0,0) rectangle (1,1)  (3,0) rectangle (4,1) (6,0) rectangle (7,1) (9,0) rectangle (10,1) (12,0) rectangle (13,1);
\drct (15,0) rectangle (16,1) (18,0) rectangle (19,1);
\drct (0,3) rectangle (1,4)  (3,3) rectangle (4,4) (6,3) rectangle (7,4) (9,3) rectangle (10,4) (12,3) rectangle (13,4);
\drct (15,3) rectangle (16,4) (18,3) rectangle (19,4);
\draw (0.5,0.5) node {1} (3.5,0.5) node {2}   (9.5,0.5) node {$j$} (15.5,0.5) node {$n\!\!-\!\!1$} (18.5,0.5) node {$n$};
\draw (0.5,3.5) node {1} (3.5,3.5) node {2}   (9.5,3.5) node {$j$} (15.5,3.5) node {$n\!\!-\!\!1$} (18.5,3.5) node {$n$};
\draw[thick,<-] (-1.8,0.5) -- (-0.6,0.5); 
\draw[thick,<-] (1.4,0.5) -- (2.6,0.5);
\draw[thick,<-] (4.4,0.5) -- (5.6,0.5); \draw[thick,<-] (7.4,0.5) -- (8.6,0.5);
\draw[thick,<-] (10.4,0.5) -- (11.6,0.5); \draw[thick,<-] (13.4,0.5) -- (14.6,0.5); 
\draw[thick,<-] (16.4,0.5) -- (17.6,0.5); 
\draw (-1.2,0.9) node {$\lambda_1$} (2,0.9) node {$\lambda_2$} (5,3.9) (8,0.9) node {$\lambda_{j}$};
\draw (11,0.9) node {$\lambda_{j+1}$} (17,0.9) node {$\lambda_{n}$};
\draw[thick,->] (-1.8,3.5) -- (-0.6,3.5); \draw[thick,->] (1.4,3.5) -- (2.6,3.5);
\draw[thick,->] (4.4,3.5) -- (5.6,3.5); \draw[thick,->] (7.4,3.5) -- (8.6,3.5);
\draw[thick,->] (10.4,3.5) -- (11.6,3.5); \draw[thick,->] (13.4,3.5) -- (14.6,3.5); 
\draw[thick,->] (16.4,3.5) -- (17.6,3.5); 
\draw (2,3.9) node {$\lambda_1q_1$} (5,3.9) node {$\lambda_2q_2$} (8,3.9) node {$\lambda_{j\!-\!1}q_{j\!-\!1}$};
\draw (11,3.9) node {$\lambda_{j}q_j$} (17,3.9) node {$\lambda_{n\!-\!1}q_{n\!-\!1}$};
\draw[thick,->] (2.6,-1.4) -- (1.6,-0.4); \draw[thick,->] (5.6,-1.4) -- (4.6,-0.4); \draw[thick,->] (8.6,-1.4) -- (7.6,-0.4);
\draw[thick,->] (11.6,-1.4) -- (10.6,-0.4); \draw[thick,->] (14.6,-1.4) -- (13.6,-0.4); \draw[thick,->] (17.6,-1.4) -- (16.6,-0.4);
 \draw[thick,->] (20.6,-1.4) -- (19.6,-0.4);
 \draw (2.5,-0.65) node {$\alpha^*_{1}$} (5.5,-0.65) node {$\alpha^*_{2}$} (11.5,-0.65) node {$\alpha^*_{j}$};
 \draw (17.7,-0.65) node {$\alpha^*_{n\!-\!1}$} (20.5,-0.65) node {$\alpha^*_{n}$};

\draw[thick,->] (1.4,4.4) -- (2.2,5.4); \draw[thick,->] (4.4,4.4) -- (5.2,5.4); \draw[thick,->] (7.4,4.4) -- (8.2,5.4);
\draw[thick,->] (10.4,4.4) -- (11.2,5.4); \draw[thick,->] (13.4,4.4) -- (14.2,5.4); \draw[thick,->] (16.4,4.4) -- (17.2,5.4);
\draw[thick,->] (19.4,4.4) -- (20.2,5.4);
\draw (2.5,4.8) node {$\lambda_{1}p_1$} (5.5,4.8) node {$\lambda_{2}p_2$} (11.5,4.8) node {$\lambda_{j}p_j$};
\draw (17.9,4.8) node {$\lambda_{n\!-\!1}p_{n\!-\!1}$} (20.3,4.8) node {$\lambda_{n}$} ;

\end{tikzpicture}

\caption{A Coxian distribution (top) and its standard reverse (bottom)}
   \label{SystRelErlang}
 \end{figure}
}

\theoremstyle{plain} 
\newtheorem{theorem}{Theorem}

\newtheorem{Proposition}[theorem]{Proposition}
\newtheorem{Corollary}[theorem]{Corollary}

\newtheorem{remark}{Remark}
\newtheorem{example}{Example}

\theoremstyle{definition} 

\newcommand{\e}{\mathrm{e}} 
\newcommand{\dd}{\mathrm{d}} 

\newcommand{\Exp}{\mathbb{E}}

\newcommand{\Prob}{\mathbb{P}}

\newcommand{\bfD}{\bm{D}}

\newcommand{\bft}{\bm{t}}
\newcommand{\bfT}{\bm{T}}
\newcommand{\bfe}{\bm{e}}

\newcommand{\bfu}{\bm{u}}
\newcommand{\bfU}{\bm{U}}

\newcommand{\bfDelta}{\bm{\Delta}}

\newcommand{\bfalpha}{{\bm \alpha}}
\newcommand{\bfnu}{\bm{\nu}}
\newcommand{\bfpi}{\bm{\pi}}
\newcommand{\1}{\bm{1}}

\newcommand{\oX}{{\overline X}}
\newcommand{\uX}{{\underline X}}
\newcommand{\osig}{{\overline \sigma}}
\newcommand{\usig}{{\underline \sigma}}
\newcommand{\op}{{\overline p}}
\newcommand{\up}{{\underline p}}
\newcommand{\oq}{{\overline q}}
\newcommand{\uq}{{\underline q}}
\newcommand{\ind}[1]{\mbox{\rm 1}_{\{#1\}}}
\newcommand{\D}{{\rm d}}

\newcommand{\topp}{{\top\hspace{-2.5mm}\top}}

\begin{document}
\begin{abstract}
This note provides a factorization of a L\'evy pocess over a phase-type horizon 
$\tau$ given the phase at the supremum, thereby extending the Wiener-Hopf factorization
for $\tau$ exponential.
One of the factors is defined using time reversal of the phase process. It is shown that there are a variety of time-reversed 
representations, all yielding the same factor.
Consequences of this are discussed and examples provided.
Additionally, some explicit formulas for the joint law of the supremum and the terminal value of the process at $\tau$ are given. 
\end{abstract}
\subjclass[2010]{Primary 60G51, 60J27}
\keywords{Wiener-Hopf factorization, splitting, time reversal, phase-type distribution}
\maketitle

\section{Introduction}
For  a non-monotone L\'evy process $X_t,t\geq 0$, define the running supremum and infimum and the times they are attained by 
\begin{gather*}\oX_t\,=\,\sup_{s\le t} X_s\,,\qquad \osig_t=\inf\{s\in[0,t]:\,X_s\vee X_{s-}=\oX_t\}\,,\\
\uX_t\,=\,\inf_{s\le t} X_s\,,\qquad \usig_t=\sup\{s\in[0,t]:\,X_s\wedge X_{s-}=\uX_t\}\,,
\end{gather*}
where by convention $X_{0-}=0$. The law of $\oX_t$ (and analogously of $\uX_t$,) plays a prominent role in applied probability, which explains an abundance of related literature.
Identification of these laws becomes much easier when~$t$ is replaced by an independent exponentially distributed random time~$\tau$, which essentially corresponds to taking Laplace transform in time. In the case of no positive jumps $\oX_\tau$ has an exponential distribution and the law of $\uX_\tau$ can be expressed in terms of the so-called scale function, which is known in explicit form in a number of subcases~\cite{scale_functions}. Moreover, one may add positive jumps of phase-type (more generally, jumps with a rational Laplace transform) while preserving tractability of the problem, cf.~\cite{asmussen_avram_pistorius, lewis_mordecki}. For further (semi-) explicit examples see~\cite{meromorphic} and references therein.

In the exponential case,
the joint law of $\oX_\tau$ and $X_\tau$ follows from the traditional Wiener-Hopf factorization, which also is an important tool in the study of the former law. It asserts that $\oX_\tau$ and $X_\tau-\oX_\tau$ are independent,
with the latter having the same distribution as $\uX_\tau$:
\begin{equation}\label{eq:WH}
\Prob\bigl(\oX_\tau\in \D x,\,X_\tau-\oX_\tau\in \D y\bigr)\ =\ 
\Prob\bigl(\oX_\tau\in \D x\bigr)\Prob\bigl(\uX_\tau\in \D y\bigr).
\end{equation}
This factorization fails for a deterministic time horizon which, nevertheless, can be remedied by conditioning on~$\osig_t$:
\begin{align}\label{eq:WH_gen}
\Prob&\bigl(\oX_t\in \D x,X_t-\oX_t\in \D y\,\big|\,\osig_t=s\bigr)\\
&=\Prob\bigl(\oX_t\in \D x\,\big|\,\osig_t=s\bigr)\Prob\bigl(\uX_t\in \D y\,\big|\,\usig_t=t-s\bigr).\nonumber\end{align}
for $s\in(0,t)$, see Appendix~\ref{sec:appendix} for further discussion. 
One may note that some intuition can be gained by proving the analogous result in the random walk setting, which is a simple exercise. The main problem of~\eqref{eq:WH_gen}, when it comes to explicit formulas, is that it requires the joint distribution of $(\oX_t,\osig_t)$, and similarly of $(\uX_t,\usig_t)$, which are presumably only available in the Brownian case treated in~\cite{Shepp}.
The aim of the present work is to state a factorization which is, in some sense, in-between~\eqref{eq:WH} and~\eqref{eq:WH_gen}. Our factorization allows for more general distributions of $\tau$ than exponential - the  dense class of phase-type distributions - but requires only the joint distribution of $\oX_\tau$ and the phase at $\osig_\tau$.    

From now on we assume that the time horizon $\tau$ is an independent random variable with a phase-type (PH) distribution, see \cite[\S III.4]{apq} or~\cite{BladtN}. By definition, $\tau$ is the life-time of some terminating time-homogenous Markov jump process $J_t,t\geq 0$ on a finite state space of phases. 
In this case it is natural to condition on the phase at $\osig_\tau$ (rather than on this time itself), which we shall show indeed leads to a factorization:
\begin{align}\label{eq:WH_ph}\Prob&\bigl(\oX_\tau\in \D x,X_\tau-\oX_\tau\in \D y\,\big |\,J_{\osig_\tau}=k\bigr)\\
&=\Prob(\oX_\tau\in \D x\,|\,J_{\osig_\tau}=k)\Prob^*(\uX_\tau\in \D y\,|\,J_{\usig_\tau}=k)\nonumber\end{align}
for any $k$, where $\Prob^*$ is a probability measure under which $X$ and $\tau$ are independent, $X$ has the law of the original L\'evy process and $J$ develops according to the standard time-reversion of~$J$; see Section~\ref{sec:factorization} for more detail. In fact, we present a more general version of this identity in Theorem~\ref{thm:ph}. To the best of our knowledge this result has not yet been stated in this form in the literature, even though Wiener-Hopf factorization of Markov additive processes (MAPs, of which we have a simple example) has been addressed at various degrees of generality in~\cite{dereich,splitting,kaspi,klusik_palmowski,kyprianou_palmowski}. The closest result is \cite[Cor.\ 5.1]{splitting} which uses a particular reversal, see Section~\ref{sec:reversal} for details, whereby for example  an Erlang distributed~$\tau$
is excluded. We remark that the factorizations~\eqref{eq:WH}, \eqref{eq:WH_gen} and~\eqref{eq:WH_ph} can be stated in greater generality by including times $\osig_\tau$ and $\tau-\osig_\tau$, or even sample paths splitting at~$\osig_\tau$; in this work we aim, however, at simple expressions rather than the utmost generality, but see Section~\ref{sec:discounting}. Throughout this work the term `splitting' refers to the decomposition of the process into two independent pieces conditional on the phase at the time of splitting.

The distribution of $(\oX_\tau,J_{\osig_\tau})$ is well known in the case when $X$ has no positive jumps and, more generally, when the positive jumps have a phase-type distribution. It is easy to see that the density of $\oX_\tau$ conditional $J_{\osig_\tau}$ can be expressed in terms of the generator of the phase process at first passage times, which is studied in detail in a number of works including~\cite{SA95,breuer,jordan_chains,dieker_mandjes,ezhov_skorokhod}.
Under the same assumption the distribution of $(X_\tau-\oX_\tau,J_\tau)$ can be expressed in terms of the matrix-valued scale function of the associated MAP~\cite{potential}, and the same is true with respect to~$(\uX_\tau,J_{\usig_\tau})$.
Furthermore, the transform $\Exp(\e^{-\theta\uX_\tau};J_{\usig_\tau}=k)$ is given in~\cite[Prop.\ 6.1]{splitting} in terms of some basic matrices, whereas some related results and calculations can be found in~\cite{dieker_mandjes,kyprianou_palmowski}.
Finally, we point out  in Section~\ref{sec:expressions} that if $X$ is a jump diffusion with phase-type jumps in both directions then~\eqref{eq:WH_ph} has a particularly simple explicit form.

Further related literature includes a multitude of papers in finance, insurance and queueing theory exploring the idea of Erlangization or Canadization, that is, approximation of a deterministic time $t$ by an Erlang time $\tau$ with $n$ phases and rates $n/t$, leading to various explicit identities. 
This idea together with the traditional Wiener-Hopf factorization is used in~\cite{kyprianou_simulation} to provide an algorithm for simulation of~$(\oX_\tau,X_\tau)$ approximating $(\oX_t,X_t)$. 

 The present research grew out of some pricing problems in life insurance considered in~\cite{GY13} and related papers.
There again the joint distribution of $X_\tau$ and $\oX_\tau$ is needed for a horizon $\tau$ that is the remaining lifetime of the insured, which by denseness may be appropriately modelled by a phase-type distribution~\cite{lin,zadeh}. The approach of~\cite{GY13} is to approximate the distribution of $\tau$
by a sum of exponential terms and use the Wiener-Hopf factorization for an exponential horizon. However, this encounters the problem
of lack of  satisfying statistical methods, so that the approach may e.g.\ lead to negative densities or tail probabilities. In contrast, for phase-type
distribution there are well developed 
statistical methods based on the EM algorithm~\cite{EMPHT},
as well as  software is available. The use of them and the factorization
in this paper will be presented in a separate paper \cite{PHLI}.

\section{The factorization}\label{sec:factorization}
\subsection{Standard reversal}
In this section, we assume $\tau$ to be PH with representation $(\bfalpha,\bfT)$.
This means that $\tau=\inf\{t\geq 0:J_t=\dagger\}$ is the life-time
of a terminating time-homogeneous Markov process $J_t$ on a finite state space with initial distribution
$\bfalpha$  (a row vector) and  generator $\bfT$, where $\dagger$ is an additional absorbing state. The associated vector of exit rates is given by $\bft=-\bfT\1$, where $\1$ is a column vector with all entries equal to 1. Without loss of generality we may and do assume that $\bfT+\bft\bfalpha$ is an irreducible matrix, because otherwise some phases can be eliminated without affecting the distribution of~$\tau$, see~\cite[Lemma 5.4.1]{BladtN}. In particular, this implies that all the non-terminating states are transient and so $T$ is invertible~\cite[Thm.\ 3.1.11]{BladtN}.

The equivalent representation $(\bfalpha^*,\bfT^*)$ of $\tau$ is obtained by considering 
\[J^*_t=\begin{cases}J_{(\tau-t)-}, &\text{ if }t<\tau,\\ \dagger, &\text{ otherwise.}\end{cases}\]
It turns out that $J^*$ is a terminating time-homogeneous Markov process with initial distribution $\bfalpha^*$ and generator $\bfT^*$ given by
\begin{equation}\label{eq:timerev}\bfalpha^*=\bft^\topp\bfDelta_{\bfnu},\quad \bfT^*=\bfDelta^{-1}_{\bfnu}\bfT^\topp\bfDelta_{\bfnu},\quad \bft^*=-\bfT^*\1=\bfDelta^{-1}_{\bfnu}\bfalpha^\topp\end{equation}
where ${}^\topp$ stands for transposition, and $\bfDelta_{\bfnu}$ is the diagonal matrix with the positive vector $\bfnu=-\bfalpha \bfT^{-1}$ on the diagonal.  
To see this consider a doubly infinite stationary Markovian arrival process 
(\cite[\S XI.1]{apq}) 
with generator $\bfT+\bfD$, where $\bfD=\bft\bfalpha$ gives the rates of phase changes with arrivals, and note that the corresponding stationary distribution is proportional to~$\bfnu$~\cite[\S III.5]{apq}. 
Then just look backwards in time. Alternatively, one may show that $J^*$ is a time-homogenous Markov chain by looking at discrete skeletons and conditioning on the epoch of absorption.

Let  $\Prob_i$ refer to the law $(X,J)$ with independent components conditional on $J_0=i$, i.e. on the initial phase being~$i$. 
We assume that $(X,J)$ under $\Prob^*$ has the same law as $(X,J^*)$ under $\Prob$. In other words, the only difference between $\Prob^*$ and $\Prob$ is that the phase process $J$ is characterized by $(\bfalpha^*,\bfT^*)$ under $\Prob^*$  instead of $(\bfalpha,\bfT)$ under $\Prob$.
We now state a more general version of~\eqref{eq:WH_ph}; from this, ~\eqref{eq:WH_ph} is obtained by summing over $j$ and over $i$ with weights~$\alpha_i$, and then dividing both sides by~$c_k$.
\begin{theorem}\label{thm:ph} The following factorization holds true whenever $\alpha_i\neq 0:$
	\begin{align}\notag
	\Prob_i&(\oX_\tau\in \D x,X_\tau-\oX_\tau\in \D y,J_{\osig_\tau}=k,J_{\tau-}=j)\\ \label{15.11a}
	&=\Prob_i(\oX_\tau\in \D x,J_{\osig_\tau}=k)\Prob_j^*(\uX_\tau\in \D y,J_{\usig_\tau}=k)\frac{\alpha^*_j}{c_k},
	\end{align}
	where $c_k=\Prob(J_{\osig_\tau}=k)=\Prob^*(J_{\usig_\tau}=k)>0$. 
\end{theorem}
\begin{proof}
We assume that $t_j\neq 0$, because otherwise both sides are 0 and there is nothing to prove.
We start with the following splitting result~\cite[Prop.~3.1]{splitting}(results of this type have a long history, see~\cite{millar_survey}):
	\begin{align*}
	&\Prob_i(\oX_\tau\in \D x,X_\tau-\oX_\tau\in \D y,J_{\osig_\tau}=k,J_{\tau-}=j)\\
	&=\Prob_i(\oX_\tau\in \D x,J_{\osig_\tau}=k)\Prob^\downarrow_k(X_{\tau}\in \D y,J_{\tau-}=j),
	\end{align*}
	where $\Prob^\downarrow_k$ is the conditional probability law  the post-supremum process $\bigl(X_{\osig_\tau+t}-\oX_{\tau},J_{\osig_\tau+t}\bigr)$ given that the second component starts in phase~$k$. It is noted that, unlike the general MAP case, there is no phase change at $\osig_\tau$ a.s. 
Similarly, there is splitting at the infimum under $\Prob^*$, and so we have by reversing time at the life-time of $J$	\begin{align*}
	&\alpha_i\Prob_i(X_\tau-\oX_\tau\in \D y,J_{\osig_\tau}=k,J_{\tau-}=j)\\
	&=\alpha^*_j\Prob^*_j(\uX_\tau\in \D y,J_{\usig_\tau}=k,J_{\tau-}=i)\\
	&=\alpha^*_j\Prob^*_j(\uX_\tau\in \D y,J_{\usig_\tau}=k){\Prob^*}^\uparrow_k(J_{\tau-}=i),
	\end{align*}
where in the first equality we also use the fact that the reversed $X$, i.e. $X_\tau-X_{(\tau-t)-}$,  has the law of the original process, because of independence of~$\tau$ (drawing a picture might be helpful here). 
Hence
\begin{equation}\label{eq1}\alpha_i\Prob_i(J_{\osig_\tau}=k)\Prob^\downarrow_k(X_{\tau}\in \D y,J_{\tau-}=j)=\alpha^*_j\Prob^*_j(\uX_\tau\in \D y,J_{\usig_\tau}=k){\Prob^*_k}^\uparrow (J_{\tau-}=i),\end{equation}
and in particular
\[\frac{\alpha_i\Prob_i(J_{\osig_\tau}=k)}{{\Prob^*_k}^\uparrow(J_{\tau-}=i)}=\frac{\alpha^*_j\Prob^*_j(J_{\usig_\tau}=k)}{\Prob^\downarrow_k(J_{\tau-}=j)}=c_k,\]
which may depend on $k$ only; it is easy to see that the denominators must be positive, because $\alpha_i,t_j>0$. Now the stated form of~$c_k$ follows by multiplying both sides with the denominator and summing over $i$ and $j$, respectively; positivity of $c_k$ is obvious from irreducibility of $\bfT+\bft\bfalpha$. Finally, from~\eqref{eq1} we find that 
\[\Prob^\downarrow_k(X_{\tau}\in \D y,J_{\tau-}=j)=\alpha^*_j\Prob^*_j(\uX_\tau\in \D y,J_{\usig_\tau}=k)/c_k\]
and the result follows. 
\end{proof}

\subsection{General reversal and examples} \label{sec:reversal}
Define $\mathcal I(\bfT)$ to be the set of all distributions $\widehat\bfalpha$ such that $\bfT+\bft\widehat\bfalpha$ is irreducible.
The result of this section is that $\Prob^*$ does not need to depend on $\bfalpha$, that is, any other $\widehat\bfalpha\in\mathcal I(\bfT)$ can be used to define the time-reversed representation $(\bfalpha^*,\bfT^*)$. To state the result, let 
$\widehat\Prob^*$ correspond to the time-reversed representation $(\widehat\bfalpha^*,\widehat\bfT^*)$ derived from $(\widehat\bfalpha,\bfT)$.

\begin{Corollary}\label{cor} Assume that $\widehat\bfalpha\in\mathcal I(\bfT)$.
Then the identity~\eqref{eq:WH_ph} continues to hold if $\Prob^*$ is replaced by $\widehat\Prob^*$. Equivalently,
\begin{equation}\label{eq:inv}
\Prob^*(\uX_\tau\in \D y|J_{\usig_\tau}=k)\ =\ \widehat\Prob^*(\uX_\tau\in \D y|J_{\usig_\tau}=k),
\end{equation}
for any~$k$.
\end{Corollary}
\begin{proof}
The key observation is that the left hand side of the identity in Theorem~\ref{thm:ph} does not depend on the initial distribution~$\bfalpha$, but only on the law of $X$ and the generator~$\bfT$. More precisely,
consider the identity of Theorem~\ref{thm:ph} for the initial distribution $\widehat\bfalpha\in\mathcal I(\bfT)$. Upon summation over $j$ it yields
\begin{align*}\MoveEqLeft \Prob_i(\oX_\tau\in \D x,X_\tau-\oX_\tau\in \D y,J_{\osig_\tau}=k)\\
&=\ \Prob_i(\oX_\tau\in \D x,J_{\osig_\tau}=k)\widehat\Prob^*(\uX_\tau\in \D y|J_{\usig_\tau}=k),\end{align*}
assuming that $\widehat\alpha_i>0$. Summing up over $i$ with weights $\alpha_i$ completes the proof for the case when $\widehat\alpha_i=0$ implies $\alpha_i=0$. This restriction can be dropped with regard to~\eqref{eq:inv} since we can start with positive $\widetilde\bfalpha$ and consider replacing it by $\bfalpha$ and by $\widehat\bfalpha$. Now the same is true about~\eqref{eq:WH_ph}.
\end{proof}

\begin{example}\label{Ex:29.9c}\rm
The relation \eqref{eq:inv} may be somewhat surprising at a first sight, but here is a simple example.
Consider the Erlang distribution with $n$ phases, i.e. $\bfalpha=\bfe_1$ and $T_{ij}=\lambda(-\ind{j=i}+\ind{j=i+1})$, and so the time-reversed representation is Erlang as well, but starting in phase~$n$. Next, we choose an arbitrary $\widehat\bfalpha$ with positive components leading to the time-reversed representation with possibility of exit in each phase, see also~\eqref{eq:relation} below.
Even though there is a possibility of early exit for the reversed phase process, the identity~\eqref{eq:inv} is still true. To see this, consider the process $X$ on the time interval restricted to the phases $n$ through $k$ under $\Prob^*$ and $\widehat\Prob^*$, and the event that $X$ attains its infimum $I$ at phase~$k$ and $I\in \D y$. Next consider the other phases and the event that $X$ does not update its infimum attained in phase $k$, which by the traditional splitting at the infimum (for an exponential time) should be independent of the other event. Now~\eqref{eq:inv} follows in this simple setup.
\end{example}

Corollary~\ref{cor}  provides a certain freedom in choosing the time-reversed representation which, however, does not lead to substantial structural changes, because of the following 
easily seen facts:
\begin{equation}\label{eq:relation}T_{ij}=0 \Leftrightarrow T^*_{ji}=0,\qquad t_i=0 \Leftrightarrow\alpha^*_i=0\qquad \alpha_i=0 \Leftrightarrow t^*_i=0,\end{equation}
see~\eqref{eq:timerev}.
Hence the most we can do, by choosing $\widehat\bfalpha$ appropriately, is to eliminate the possibility of exit from certain states in the reversed representation.
\begin{example}\label{Ex:9.11a}\rm
Consider a generalized Coxian distribution with an arbitrary initial distribution $\bfalpha$ and generator given by $T_{ij}=-\lambda_i\ind{j=i}+\lambda_i(1-p_i)\ind{j=i+1}$, where $\lambda_i>0$ and $p_i=1-q_i\in[0,1)$ for $i<n$ and $p_n=1$. In fact, we need to assume that $\alpha_1>0$ otherwise $\bfalpha\notin\mathcal I(\bfT)$. The reversed representation simplifies for $\widehat\bfalpha=\bfe_1$ (standard Coxian) yielding 
$\widehat T^*_{ij}=-\lambda_i\ind{j=i}+\lambda_i\ind{j=i-1}$, and $\widehat\alpha^*_i=p_i(1-p_{i-1})\cdots(1-p_1)$, so that termination always occurs in phase~1, see Figure~\ref{SystRelErlang}.
\end{example}
 \CoxFig
 
\subsection{Remarks on alternative reversals}
\begin{remark}\label{remark:alt_rev}\rm
A seemingly even more general reversal is obtained by considering a 
Markovian Arrival Process with generator $\bfT+\bfD$, where $\bfD$ is not necessarily of the form $\bft\widehat\bfalpha$.
 Provided $\bfT+\bfD$ is irreducible, any $\bfD$ satisfying
 $\bfD\1=\bft$ will do. This, however, does not lead to more general backward representations. 
 Let $\widehat\bfalpha$ be the corresponding event stationary distribution, which then must be proportional to $\bfpi \bfD$~\cite[Prop.\ XI.1.4]{apq}, where $\bfpi$ is the stationary distribution of $\bfT+\bfD$.
 Then $\bfpi$ is also stationary for $\bfT+\bft\widehat\bfalpha$, as can be easily verified, and indeed this reversal is the same as the standard reversal for the initial distribution $\widehat\bfalpha$. Note also that irreducibility of $\bfT+\bfD$ implies that of $\bfT+\bft\widehat\bfalpha$.
 \end{remark}

\begin{remark}\label{Rem:5.10a}\rm
Suppose we require that $\bfalpha=\bfalpha^*$, 
which is the same as $-\bfalpha \bfT^{-1}\bfDelta_{\bft}$ $=\bfalpha$ or, equivalently, $-\bfalpha \bfT^{-1}(\bfT+\bfDelta_{\bft})=\bf 0$. Assume that the matrix $\bfT+\bfDelta_{\bft}$ is irreducible, and so there is a unique stationary distribution $\bfpi$, i.e.\ such that $\bfpi(\bfT+\bfDelta_{\bft})=\bf 0$. Then there exists a positive constant $c$ such that $\bfnu=-\bfalpha \bfT^{-1}=\bfpi/c$. It is now clear that 
\[\bfalpha^*=\bfalpha=\bfpi\bfDelta_{\bft}/(\bfpi\bft), \qquad \bfT^*=\bfDelta^{-1}_{\bfpi}\bfT^\topp\bfDelta_{\bfpi}, \qquad \bft^*=\bft.\]
The same conclusion is obtained if one starts by requiring that $\bft^*=\bft$. Note that here we are in the setting of Remark~\ref{remark:alt_rev} with $\bfD=\bfDelta_{\bft}$.
This time-reversal can be seen as an outcome of the following procedure: (i) remove killing, (ii) reverse the recurrent process, (iii) add back the killing rates.
It is the one used exclusively in~\cite{splitting}, but any other reversal would work there as well. In this regard, note that $\bfT+\bfDelta_{\bft}$ is not irreducible for the Erlang distribution.
\end{remark}

\section{Special cases}
\subsection{Jump diffusion with PH jumps}\label{sec:expressions}
%
%
%
%
Let $X_t$ be a jump diffusion with both up- and downward jumps having PH distributions:
\[X_t=B_t+\sum_{i=1}^{N^+_t} Y_i^+-\sum_{i=1}^{N^-_t} Y_i^-,\]
where $B$ is a  Brownian motion with drift $\mu$ and variance $\sigma^2>0$, $N^\pm$ are Poisson processes and $Y_i^\pm$ have the respective PH distribution for all~$i$; all components being independent. We assume that these PH distributions have representations with $n^\pm$ phases $ 1^+,\ldots,n^+$, $ 1^-,\ldots, n^-$
(but not less), whereas $J$ lives on phases $1,\ldots,n$.
A standard procedure (e.g.\ \cite[Section XI.1c]{apq}) is to level out jumps, thereby providing a
fluid embedding in a terminating Markov-modulated Brownian motion
with $n(1+n^++n^-)$ phases $i,ij^\pm$, $i=1,\ldots,n$, $j=1,\ldots,n^\pm$,
such that the drift parameters in 
the different states are $\mu_i=\mu$, $\mu_{i j^+}=1$,  $\mu_{i j^-}=-1$, 
the only none-zero variances are the $\sigma_i^2=\sigma^2$, and the killing rate is
$t_i$ in state $i$, 0 in all $ij^\pm$  states.
The total maximum of this process then equals 
$\oX_\tau$. Let $\bfU$ and $\bfU^*$ be the generators of
the Markov chain formed by the phase at first passage over positive levels, resp.\ negative levels,
with $n(n^++1)$ and $n(n^-+1)$ rows. Then
\begin{align*}
\Prob_i\bigl(\oX_\tau\in\dd x,\, J_{\osig_\tau}=k\bigr)\ &=\ \left(\e^{\bfU x}\right)_{ik}u_k\D x\\ 
\Prob_j^*(\uX_\tau\in \D y,J_{\usig_\tau}=k)\ &=\ \left(\e^{-\bfU^* y}\right)_{jk}u^*_k\D y,
\end{align*}
where $\bfu=-\bfU\1,\bfu^*=-\bfU^*\1$, $x,-y\geq 0$. 
Here the matrices $\bfU$ and $\bfU^*$  are numerically computable  using either iterative or spectral methods, see~\cite[Thm.\,4.1]{SA95} (or several later sources) and \cite{jordan_chains}.
Our factorization \eqref{15.11a} then takes the following form:
\begin{Corollary}\label{Cor:15.11b} Define $r_k=u_ku_k^*/c_k=u_k^*/(-\bfalpha \bfU^{-1})_k=u_k/(-\bfalpha^* {\bfU^*}^{-1})_k$,
where $\bfalpha$ and $\bfalpha^*$ are  extended to the appropriate size by trailing zeros.
 Then
\begin{align*}\MoveEqLeft  \Prob_i(\oX_\tau\in \D x,X_\tau-\oX_\tau\in \D y,J_{\osig_\tau}=k,J_{\tau-}=j)\\ & =\ 
r_k\alpha^*_j\cdot\left(\e^{\bfU x}\right)_{ik}\D x\cdot \left(\e^{-\bfU^* y}\right)_{jk}\D y.
\end{align*}
\end{Corollary}
Again the left hand side does not depend on the initial distribution $\bfalpha$, and so we may use any $\bfalpha\in\mathcal I(\bfT)$ to define the reversed quantities $\bfalpha^*$ and~$\bfU^*$.

\subsection{Brownian motion with Erlang distributed time horizon}

Assume that $X$ is BM$(\mu,\sigma^2)$ and that $\tau$ has Erlang distribution with $n$ stages and rate parameter $\lambda$; we write ${E}_{n,\lambda}$ for a random variable with such law. We show that in this case
the factors in~\eqref{eq:WH_ph} take a simple form as mixtures of Erlangs with easily computed coefficients:

\begin{Proposition}\label{prop:BME} For the BM-Erlang case,
\begin{align}\label{BME1}\Prob(\oX_\tau\in\dd x,J_{\osig_\tau}=k)\ &=\ \uq(n-k+1)\sum_{i=1}^k\op(i;\,k)
\Prob\bigl(E_{i,\lambda_+}\in\,\dd x\bigr) \,,\\ 
\label{BME2}
\Prob(-\uX_\tau\in\dd x,J_{\usig_\tau}=k)\ &=\ \oq(n-k+1)\sum_{i=1}^k \up(i;\,k)
\Prob\bigl(E_{i,\lambda_-}\in\,\dd x\bigr),
\end{align}
where 
\begin{align}\label{eq:rest}
\lambda_{\pm}\ =\ \mp\frac{\mu}{\sigma^2}+\sqrt{\frac{\mu^2}{\sigma^4}+\frac{2\lambda}{\sigma^2}},\qquad
\oq(k)= \sum_{i=1}^k \op(i;\,k)\,,\quad 
\uq(k)= \sum_{i=1}^k \up(i;\,k),
\end{align}
and $\op(i;\,k),\up(i;\,k)$ can be computed recursively from $\op(1;1)=\up(1;1)=1$ and~\eqref{12.12f}, \eqref{12.12g} below.
\end{Proposition}
\begin{proof}
Write  $\oX_k,\uX_k,\osig_k$
etc.\ for $\oX,\uX,\osig$ computed from the first $k$ stages. 
It is standard that for the exponential case $k=1$, $\oX_1$ and $-\uX_1$ are exponentials with rates $\lambda_\pm$ given in~\eqref{eq:rest}. Moreover, the path splits at the supremum.
Ladder height arguments then give that the distribution of the maximum $\oX_k$ 
is a mixture of $E_{1,\lambda_+},\ldots,E_{k,\lambda_+}$ and,
more generally, that the   same is true for the defective distribution
obtained when one restricts to the set where the maximum 
was attained in stage $k$, i.e.\ $J_{\osig_k}=k$.
Similar observations are in place  for the minimum $-\uX_k$. Denote the corresponding weights of the Erlangs by $\op(i;\,k),\up(i;\,k)$, $i=1,\ldots,k$.
If $J_{\osig_k}=k$, then $J_{\osig_n}=k$ will occur precisely if the reversed sample path (in time and space) restricted to stages $n,\ldots,k$ has its minimum in the last stage~$k$. This event is independent of $(\oX_k,J_{\osig_k})$, and has probability $\underline q(n-k+1)$. Hence the formulas~\eqref{BME1} and~\eqref{BME2} follow. 

It is left to identify $\op(i;\,k)$ and $\up(i;\,k)$. Clearly, $\op(1;1)=\up(1;1)=1$.
Consider the probabilities that independent exponentials of rates $\lambda_\pm$ exceed one another: 
\begin{equation}\label{12.12d}\theta_+=\frac{\lambda_-}{\lambda_++\lambda_-}\,,\ \ 
 \theta_-=\frac{\lambda_+}{\lambda_++\lambda_-}.\end{equation}
 The probability that an independent exponential of rate $\lambda_+$ exceeds $E_{j,\lambda_-}$ is then $\theta_+^j$. 
 To update from $k-1$ to $k\ge 2$, let $\ell\le k-1$ be the stage in which
$\oX_{k-1}$ is attained
and use the same reversal argument as above for stages $\ell,\ldots,k-1$ 
combined with the structure
of $-\uX_{k-\ell}$ as a mixture of Erlang$(j,\lambda_-)$'s, which must be exceeded by an independent exponential of rate $\lambda_+$ corresponding to the stage~$k$. This yields
\begin{equation}\label{12.12f}\op(i;\,k)\ =\ \sum_{\ell=i-1}^{k-1} \op(i-1;\,\ell)\sum_{j=1}^{k-\ell} \up(j;\,k-\ell)\theta_+^j
\end{equation}
for $i=2,\ldots,k$
(by direct inspection $\op(1;\,k)=0$ for $k\ge 2$). 
Here  the Erlang weights  $\up(j;\,k-\ell)$
have been computed in earlier steps since $k-\ell\le k-1$. Similarly, 
\begin{equation}\label{12.12g}\up(i;\,k)\ =\ \sum_{\ell=i-1}^{k-1} \up(i-1;\,\ell)\sum_{j=1}^{k-\ell} \op(j;\,k-\ell)\theta_-^j
\end{equation}
 which completes the
proof.
\end{proof}

Standard reversal of Erlang yields the same Erlang with the opposite ordering of phases. Hence by combining Theorem~\ref{thm:ph} with Proposition~\ref{prop:BME} we get
\begin{align*}
&\Prob(\oX_\tau\in \D x,\oX_\tau-X_\tau\in \D y,J_{\osig_\tau}=k)\\
	&=\sum_{i=1}^k\op(i;\,k)
\Prob\bigl(E_{i,\lambda_+}\in\,\dd x\bigr)\times\sum_{i=1}^{n-k+1} \up(i;\,n-k+1)
\Prob\bigl(E_{i,\lambda_-}\in\,\dd y\bigr),
\end{align*}
which also follows directly using the above reasoning in this simple case.

\section{Discounting}\label{sec:discounting}
In various application, for example, in finance and insurance, it is often required to compute expected discounted quantities. 
To address this issue we state a slightly more general identity than in Theorem~\ref{thm:ph}: for any $\delta\geq 0$ it holds that
\begin{align}\notag
	\Exp_i&\left(\e^{-\delta \tau};\oX_\tau\in \D x,X_\tau-\oX_\tau\in \D y,J_{\osig_\tau}=k,J_{\tau-}=j\right)\\ \label{15.11a}
	&=\Exp_i\left(\e^{-\delta \osig_\tau};\oX_\tau\in \D x,J_{\osig_\tau}=k\right)\Exp_j^*\left(\e^{-\delta \usig_\tau};\uX_\tau\in \D y,J_{\usig_\tau}=k\right)\frac{\alpha^*_j}{c_k},
	\end{align}
where the rest (including reversal and the constants $c_k$) is the same as before. The only change in the proof is that we track the time before and after splitting. Hence we also have the obvious extension of~\eqref{eq:WH_ph} with the above stated freedom in choosing the reversed representation, see Section~\ref{sec:reversal}. Of course, the first factor identifies the joint law of $(\overline X_\tau,\osig_\tau)$, but nevertheless it can be given in explicit form in various cases. Below we specialize to the setting of Section~\ref{sec:expressions}, i.e., the  jump diffusion with PH jumps in both directions.

Let $\bfU(\delta)$ be the analogue of $\bfU$ obtained by additional killing the terminating Markov-modulated Brownian motion at rate $\delta$ in the phases $1,\ldots,n$ corresponding to the Brownian evolutions. Such matrix $\bfU(\delta)$ leads to the factor for the PH time $\tau\wedge e_\delta$, where $e_\delta$ is an independent exponential time of rate~$\delta$. 
Now it is left to ensure that the killing of the first passage Markov chain is due to $\tau$, which results in the identity
\[\Exp_i\left(\e^{-\delta \osig_\tau};\oX_\tau\in\dd x,\, J_{\osig_\tau}=k\right)\ =\ \left(\e^{\bfU(\delta) x}\right)_{ik} u_k\D x.\]
Note that in this formula we use $u_k$ and not $u_k(\delta)$. In particular, we get the following generalization of Corollary~\ref{Cor:15.11b}:
\begin{align}\MoveEqLeft  \Exp_i\left(\e^{-\delta\tau};\oX_\tau\in \D x,X_\tau\in \D y,J_{\osig_\tau}=k,J_{\tau-}=j\right)\nonumber\\ & =\ 
r_k\alpha^*_j\cdot \left(\e^{\bfU(\delta) x}\right)_{ik}\D x\cdot \left(\e^{\bfU^*(\delta) (x-y)}\right)_{jk}\D y
\end{align}
for all $x>0,x>y$, where the constants $r_k$ stay unchanged.

  \appendix
  \section{Formal proof of~\eqref{eq:WH_gen}}\label{sec:appendix}
  The relation~\eqref{eq:WH_gen} is intuitively clear, but in the following we provide a formal proof. 
  It is well-known that $\osig_t$ has a density on $(0,t)$ and so~\eqref{eq:WH_gen} should be understood in the sense of regular conditional distributions. 
  For a non-monotone L\'evy process $X$,~\cite[Thm.\ 6]{chaumont} 
  gives that
  \[\Prob\bigl(\oX_t\in \D x,\oX_t-X_t\in \D y,\osig_t\in \D s\bigr)=\underline n(X_s\in\D x)\overline n(X_{t-s}\in\D y)\D s,\]
 for $s$ restricted to $(0,t)$, where $\overline n$ and $\underline n$ are the It\^o measures of the excursions from the supremum and infimum, respectively. Integrating over $x$ and $y$ we find that $\osig_t$ has a density $\underline n(\zeta>s)\overline n(\zeta>{t-s})$, where $\zeta$ is the life time of the generic excursion. Hence there is a factorization:
 \begin{align*}\Prob&\bigl(\oX_t\in \D x,\oX_t-X_t\in \D y|\osig_t=s\bigr)=\frac{\underline n(X_s\in\D x)}{\underline n(\zeta>s)}\frac{\overline n(X_{t-s}\in \D y)}{\overline n(\zeta>t-s)}\\
 &=\Prob\bigl(\oX_t\in \D x|\osig_t=s\bigr)\Prob\bigl(\oX_t-X_t\in \D y|\osig_t=s\bigr).\end{align*}
But by the standard reversal of $X$ at time $t$ we find
\[\Prob\bigl(\oX_t-X_t\in \D y,\osig_t\in \D s\bigr)=\Prob\bigl(-\uX_t\in \D y,\usig_t\in \D (t-s)\bigr),\]
 which completes the proof of~\eqref{eq:WH_gen}.

\end{document}